\documentclass[a4paper,11pt]{amsart}
\usepackage{amssymb}
\usepackage{amsmath}
\usepackage{amsthm}
\usepackage{amscd}
\usepackage{verbatim}
\usepackage{graphicx}
\usepackage[all]{xy}

\usepackage{url}
\usepackage{color}

\DeclareFontEncoding{OT2}{}{} 
\DeclareTextFontCommand{\textcyr}{\fontencoding{OT2}
    \fontfamily{wncyr}\fontseries{m}\fontshape{n}\selectfont}

\theoremstyle{plain}

\newtheorem{theorem}{Theorem}[section]

\newtheorem{lemma}[theorem]{Lemma}
\newtheorem{corollary}[theorem]{Corollary}

\theoremstyle{definition}

\newtheorem{definition}[theorem]{Definition}

\newtheorem{construction}[theorem]{Construction}
\newtheorem{notation}[theorem]{Notation}

\title[Real Galois cohomology]
{Galois cohomology of real semisimple groups
}

\author{Mikhail Borovoi and Dmitry A. Timashev}

\address{Borovoi: Raymond and Beverly Sackler School of Mathematical Sciences,
Tel Aviv University, 6997801 Tel Aviv, Israel}
\email{borovoi@post.tau.ac.il}

\address{Timashev:
Lomonosov Moscow State University, Faculty of Mechanics and
Mathematics, Department of Higher Algebra, 119991 Moscow, Russia}
\email{timashev@mech.math.msu.su }

\thanks{Borovoi was partially supported
by the Hermann Minkowski Center for Geometry}


\begin{document}


\def\noi{\noindent}
\def\R{{\mathbb{R}}}
\def\Z{{\mathbb{Z}}}
\def\C{{\mathbb{C}}}
\def\Q{{\mathbb{Q}}}
\def\F{{\mathbb{F}}}
\def\Xt{{\tilde{X}}}
\def\id{{{\rm id}}}
\def\diag{{\rm diag}}
\def\gcd{{{\rm gcd}}}
\def\lcm{{{\rm lcm}}}
\newcommand{\Img}{\operatorname{Im}}

\def\AA{{\mathbf{A}}}
\def\BB{{\mathbf{B}}}
\def\CC{{\mathbf{C}}}
\def\DD{{\mathbf{D}}}
\def\EE{{\mathbf{E}}}
\def\FF{{\mathbf{F}}}
\def\GG{{\mathbf{G}}}
\def\HH{{\mathbf{H}}}
\def\NN{{\mathbf{N}}}
\def\ZZ{{\mathbf{Z}}}
\def\WW{{\mathbf{W}}}
\def\TT{{\mathbf{T}}}

\def\ad{{\rm ad}}

\newcommand{\X}{{{\sf X}}}

\def\Hom{{\rm Hom}}
\def\G{{\mathbb{G}}}

\def\into{\hookrightarrow}
\newcommand{\isoto}{\overset{\sim}{\to}}

\def\aa{{\boldsymbol{a}}}
\def\bb{{\boldsymbol{b}}}
\def\cc{{\boldsymbol{c}}}
\def\dd{{\boldsymbol{d}}}
\def\ttt{{\mathfrak{t}}}
\def\sss{{\boldsymbol{s}}}
\def\mmu{{\boldsymbol{\mu}}}
\def\nnu{{\boldsymbol{\nu}}}
\def\pp{{\boldsymbol{p}}}
\def\qq{{\boldsymbol{q}}}

\def\arr{\ar@{=>}[r]}
\def\vk{{k}}

\def\zbar{{\overline{z}}}
\def\gbar{{\overline{g}}}
\def\nbar{{\overline{n}}}

\def\ab{{\rm ab}}
\def\Gtil{{\widetilde{G}}}

\newcommand{\labelto}[1]{\xrightarrow{\makebox[1.5em]{\scriptsize ${#1}$}}}

\newcommand{\redu}{\mathrm{red}}
\newcommand{\sem}{\mathrm{ss}}
\newcommand{\scon}{\mathrm{sc}}
\newcommand{\tor}{\mathrm{tor}}

\newcommand{\GL}{{\bf{GL}}}
\newcommand{\SL}{{\bf{SL}}}
\newcommand{\Sp}{{\bf{Sp}}}
\newcommand{\PSp}{{\bf{PSp}}}
\newcommand{\SO}{{{\bf SO}}}
\newcommand{\PSO}{{\bf{PSO}}}
\newcommand{\Spin}{{{\bf Spin}}}
\newcommand{\HSpin}{{\bf{HSpin}}}
\newcommand{\PGL}{{\bf{PGL}}}
\newcommand{\SU}{{\bf SU}}
\newcommand{\PSU}{{\bf PSU}}

\def\Inn{{\rm Inn}}
\def\Aut{{\rm Aut}}

\newcommand{\boxone}{ *+[F]{1} }
\def\ccc{{  \lower0.3ex\hbox{{\text{\Large$\circ$}}}}}
\def\circledS{{\lower0.31ex\hbox{\text{\Large$\bullet$}}}}
\newcommand{\bc}[1]{{\overset{#1}{\ccc}}}
\newcommand{\bcu}[1]{{\underset{#1}{\ccc}}}
\newcommand{\bcb}[1]{{\overset{#1}{\circledS}}}
\newcommand{\bcbu}[1]{{\underset{#1}{\circledS}}}
\newcommand{\sxymatrix}[1]{ \xymatrix@1@R=5pt@C=10pt{#1} }
\newcommand{\vxymatrix}[1]{ \xymatrix@1@R=15pt@C=10pt{#1} }
\newcommand{\rline}{ \ar@{-}[r] }
\newcommand{\lline}{ \ar@{-}[l] }
\newcommand{\dline}{ \ar@{-}[d] }
\newcommand{\upline}{ \ar@{-}[u] }

\newcommand{\mxymatrix}[1]{ \xymatrix@1@R=0pt@C=9pt{#1} }

\def\gSU{{\bf SU}}
\def\AA{{\bf A}}

\def\nbar{{\bar{n}}}

\def\Inn{{\rm Inn}}
\def\Lie{{\rm Lie\,}}

\def\bb{{\boldsymbol{b}}}

\def\Htil{{\widetilde{H}}}
\def\Ttil{{\widetilde{T}}}
\def\im{{\rm im\,}}

\def\emm{\bfseries}

\def\Gal{{\rm Gal}}
\def\H{{\mathbb{H}}}

\def\ppi{{\boldsymbol{\pi}}}

\def\RRR{{\Rightarrow}}
\def\gr{\! > \!}
\def\less{ \!\! < \!\! }
\def\LLL{\!\Leftarrow\!}
\def\boe{{\sxymatrix{\boxone}}}
\def\ll{\!-\!}

\def\jj{{\boldsymbol{j}}}
\def\ii{{\boldsymbol{i}}}

\def\OO{{\mathbf{O}}}
\def\ve{\varepsilon}

\def\sV{{\mathcal{V}}}
\def\xbar{{\bar{x}}}
\def\U{{\bf U}}
\def\sH{{\mathcal{H}}}
\def\sK{{\mathcal{K}}}

\def\vev{\varepsilon^\vee}
\def\onto{\twoheadrightarrow}

\def\gg{{\mathfrak{g}}}
\def\Ad{{\rm Ad}}

\def\coker{{\rm coker\,}}
\def\tors{{\rm tors}}
\def\Ext{{\rm Ext}}
\def\Stab{{\rm Stab}}

\def\half{{\tfrac{1}{2}}}
\def\nth{{\tfrac{1}{n}}}

\def\Ss{\sideset{}{'}\sum_{k\succeq i}}
\def\Sst{\Ss}
\def\Ssd{\sideset{}{'}\sum_{k\succeq i,\,k\in D^\tau}}

\def\hs{\kern 0.8pt}
\def\hh{\kern 1.0pt}

\def\sc{{\rm sc}}
\def\Dtil{{\widetilde{D}}}

\def\even{{\rm even}}
\def\odd{{\rm odd}}
\def\Orb{{\rm Orb}}

\newcommand{\xx}{{\bf x}}

\def\Pitil{{\widetilde{\Pi}}}
\def\vvk{{(\vk)}}

\def\ds{\displaystyle}

\newcommand\two[2]{\underset{\ds#2}{#1}}
\newcommand\three[3]{\underset{\ds#3}{\two{#1}{#2}}}
\newcommand\four[4]{\underset{\ds#4}{\three{#1}{#2}{#3}}}
\newcommand\five[5]{\underset{\ds#5}{\four{#1}{#2}{#3}{#4}}}

\newcommand{\gf}[2]{\genfrac{}{}{0pt}{}{#1}{#2}}



\begin{abstract}
Let $\GG$ be a connected, compact, semisimple algebraic group over the field of real numbers $\R$.
Using Kac diagrams, we describe combinatorially the first Galois cohomology sets $H^1(\R,\HH)$ for all inner forms $\HH$ of $\GG$.
As examples, we compute explicitly $H^1$ for all real forms of the simply connected simple group of type $\EE_7$ (which has been known since 2013)
and for all real forms of half-spin groups of type $\DD_{2k}$ (which seems to be new).
\end{abstract}

\maketitle


\setcounter{section}{-1}

\section{Introduction}

Let $\HH$ be a linear algebraic group defined over the field of real numbers $\R$.
For the definition of the  first (nonabelian) Galois cohomology set $H^1(\R,\HH)$ see Section \ref{s:cohomology} below.
Galois cohomology can be used to answer many natural questions (on classification of real forms, on the connected components
of the set of $\R$-points of a homogeneous space etc.).
The  Galois cohomology sets $H^1(\R,\HH)$ of the classical groups are well known.
Recently the sets $H^1(\R,\HH)$ were computed for ``most'' of the {\em simple} $\R$-groups  by Adams \cite{A},
in particular, for all {\em simply connected} simple $\R$-groups by Adams \cite{A} and by Borovoi and Evenor \cite{BE}.

Victor G. Kac \cite{Kac} used what was later called Kac diagrams
(see Onishchik and Vinberg \cite[Sections 3.3.7 and 3.3.11]{OV2}\hs) to classify the conjugacy classes of automorphisms of finite order
of a simple Lie algebra over the field of complex numbers $\C$.
Let $\GG$ be a {\em compact} (anisotropic), simply connected,  simple algebraic group over $\R$.
Write $\GG_\C=\GG\times_\R \C$, $\gg_\C=\Lie(\GG_\C)$.
With this notation, Kac classified the  conjugacy classes of elements of order $n$ in $\Aut\,\gg_\C=\Aut\,\GG_\C$.
In particular, he classified the conjugacy classes of elements of order $n$ in the group of inner automorphisms $\GG^\ad(\C)\subset \Aut\,\GG_\C$,
where $\GG^\ad:=\GG/\ZZ_\GG$ is the corresponding adjoint group.
Equivalently,  he classified  the conjugacy classes of elements of order $n$ in $\GG^\ad(\R)$.

Note that the set of conjugacy classes of elements of order $n=2$ in $\GG^\ad(\R)$ is in canonical bijection with the first Galois cohomology set
$H^1(\R,\GG^\ad)$, see Serre \cite[Section III.4.5, Theorem 6]{Serre}.
Thus Kac computed $H^1(\R, \GG^\ad)$, the Galois cohomology of the compact, simple, {\em adjoint} $\R$-group $\GG^\ad$.

In the present paper we use the method of Kac diagrams in order to compute $H^1(\R,\GG)$, or more generally $H^1(\R,\hs_\qq \GG)$,
where $\GG$ is a connected, compact, {\em semisimple} $\R$-group, not necessarily adjoint,
and $_\qq \GG$ is the {\em inner} twisted form of $\GG$ corresponding to a Kac diagram $\qq$.
This is reduced to classifying conjugacy classes of square roots of a given central element $z=z_\qq\in\GG(\R)$.

The plan of the paper is as follows.
In Section \ref{s:notation} we introduce the necessary notation.
In Section \ref{s:action} we describe, following Bourbaki \cite{Bourbaki}, the action of  $P^\vee/Q^\vee$ on the extended Dynkin diagram of a root system $R$,
where $P^\vee$ is the coroot lattice and $Q^\vee$ is the coweight lattice.
The heart of the paper is Section \ref{s:roots}, where we prove Theorem \ref{thm:p-z} describing the conjugacy classes of $n$-th roots of a given central element $z$
in a connected semisimple compact Lie group $G$ in terms of certain combinatorial objects called {\em Kac $n$-labelings of the extended Dynkin diagram $\Dtil$ of $G$.}
Using this theorem (in the case $n=2$) and a result of  \cite{Bo}, in Section \ref{s:cohomology} we prove  \  Theorem \ref{t:main}, which is the main result of this paper.
It describes the  first Galois cohomology set $H^1(\R,\hs_\qq \GG)$
of an  inner twisted form $_\qq \GG$ of a connected compact (anisotropic) semisimple $\R$-group $\GG$ in terms of  Kac 2-labelings.
As an example,  in Section \ref{s:E7} we compute, using Kac 2-labelings,  the Galois cohomology sets $H^1(\R,\,_\qq \GG)$
for all $\R$-forms $_\qq \GG$ of the compact simply connected group $\GG$ of type $\EE_7$;
these results  were obtained earlier by other methods in \cite{A} and  \cite{BE}, see also Conrad \cite[Proof of Lemma 4.9]{Conrad}.
As another example,  in Section \ref{s:half-spin} we compute the Galois cohomology sets $H^1(\R,\,_\qq \GG)$
for all  $\R$-forms of a half-spin compact group of type $\DD_{\ell}$ for even $\ell>4$;
these results seem to be new.

The authors are grateful to E.\,B.~Vinberg, whose for e-mail correspondence with the first-named author in 2008 inspired this paper.

\section{Notation}
\label{s:notation}

In this paper $\GG$ always is a connected, compact (anisotropic), semisimple algebraic group over the field of real numbers $\R$.
We write $\ZZ_\GG$ for the center of $\GG$.
Let $\GG^\ad=\GG/\ZZ_\GG$ denote the corresponding adjoint group, and let $\GG^\sc$
denote the universal covering of $\GG$ (which is simply connected).
Let $\TT\subset\GG$ be a maximal torus.
We denote by $\ttt$ the Lie algebra of $\TT$, which is a vector space over $\R$.
Let $\NN=\mathcal{N}_\GG(\TT)$  denote the normalizer of $\TT$ in $\GG$.
Let $\WW=\NN/\TT$ be the Weyl group, which is a finite algebraic group.

Let $\TT^\ad:=\TT/\ZZ_\GG$ be the image of $\TT$ in $\GG^\ad$, and let $\TT^\sc$ denote the preimage of $\TT$ in $\GG^\sc$.
Then $\TT^\ad$ is a maximal torus in $\GG^\ad$, and $\TT^\sc$ is a maximal torus in $\GG^\sc$.
Set
\[ X=\X(\TT_\C):=\Hom(\TT_\C,\G_{m,\C}),\quad  X^\vee=\X^\vee(\TT_\C):=\Hom(\G_{m,\C},\TT_\C), \]
where $\TT_\C=\TT\times_\R \C$ and $\G_{m,\C}$ is the multiplicative group over $\C$; then $X$ and $X^\vee$ are the character group and the cocharacter group of $\TT_\C$, respectively.

We have a canonical isomorphism of abelian complex Lie groups
$$
X^\vee\underset{\Z}{\otimes} \C^\times\isoto \TT(\C),\quad \chi\otimes u\mapsto \chi(u),\quad \chi\in X^\vee,\ u\in \C^\times=\G_{m,\C}(\C).
$$
Thus we obtain an isomorphism of abelian complex Lie algebras (vector spaces over $\C$)
\begin{align*}
X^\vee \underset{\Z}{\otimes} \C \isoto \Lie \TT_\C,\quad \chi\otimes v\mapsto d\chi(v),\quad
&\chi\in X^\vee,\ v\in\C, \\
&d\chi:=d_1\chi\colon \C=\Lie\G_{m,\C}\to \Lie \TT_\C\,.
\end{align*}
We obtain the standard embedding
\begin{equation*}
X^\vee\into X^\vee \underset{\Z}{\otimes} \C \isoto \Lie \TT_\C\,, \quad
\chi\mapsto \chi\otimes 1\mapsto d\chi(1).
\end{equation*}

As usual, we set
\[ P=\X(\TT^\sc_\C),\quad Q=\X(\TT^\ad_\C); \]
these are the weight lattice and the root lattice.
We set also
\[ P^\vee=\X^\vee(\TT^\ad_\C),\quad Q^\vee=\X^\vee(\TT^\sc_\C); \]
these are the coweight lattice and the coroot lattice.
Then
\[ Q\subset X\subset P\quad\text{and} \quad Q^\vee\subset X^\vee\subset P^\vee. \]

Let $\GG$ and $\TT$ be as above.
We write $G=\GG(\R)$ for the set of $\R$-points of $\GG$, and similarly we write $G^\ad=\GG^\ad(\R)$, $G^\sc=\GG^\sc(\R)$.
We write $T=\TT(\R)$, and similarly we write $T^\ad=\TT^\ad(\R)$, $T^\sc=\TT^\sc(\R)$.
We write $N=\NN(\R)$ and  $W=\WW(\R)$.
We write $Z_G=\ZZ_\GG(\R)$ for the center of $G$.

We define an action of the group $X^\vee\rtimes W$ on the set $\ttt$ as follows:
an element $\chi\in X^\vee\subset \ttt_\C$ acts by translation by $\ii\hs \chi\in\ttt$
(where $\ii^2=-1$),
and $w\in W\subset\Aut\, \TT$ acts on  $\ttt=\Lie \TT$ as usual, i.e., as $d_1 w\colon \Lie \TT\to\Lie \TT$.
It follows that the groups $Q^\vee\rtimes W$ and $P^\vee\rtimes W$ act on $\ttt$.

Let $R=R(\GG_\C,\TT_\C)$ denote the root system of $\GG_\C$ with respect to $\TT_\C$.
Let $\Pi\subset R$ be a basis (a system of simple roots).
Let $D=D(\GG,\TT,\Pi)=D(R,\Pi)$ denote the Dynkin diagram;
the set of the vertices  of $D$ is $\Pi$.

Assume that $\GG$ is (almost) simple.
We write $\Pi=\{\alpha_1,\dots,\alpha_\ell\}$.
Let $\Dtil=\Dtil(\GG,\TT,\Pi)=\Dtil(R,\Pi)$  denote the extended Dynkin diagram;
the set of vertices of $\Dtil$ is   $\Pitil=\{\alpha_1,\dots,\alpha_\ell,\alpha_0\}$,
where $\alpha_1,\dots,\alpha_\ell$ are the simple roots, and $\alpha_0$ is the lowest root.
These roots $\alpha_1,\dots,\alpha_\ell, \alpha_0$ are linearly dependent, namely,
\begin{equation}\label{eq:sum-mj}
m_{\alpha_1}\alpha_1+\dots+m_{\alpha_\ell}\alpha_\ell+m_{\alpha_0}\alpha_0=0,
\end{equation}
where the coefficients $m_{\alpha_j}$ are positive integers for all $j=1,\dots,\ell,0$ and where $m_{\alpha_0}=1$.
We write $m_j$ for $m_{\alpha_j}$.
These coefficients $m_j$ are tabulated in \cite[Table 6]{OV} and in \cite[Table 3]{OV2}.

Now assume that $\GG$ is semisimple, not necessarily simple.
Then we have a decomposition $\GG=\GG^{(1)}\cdot \GG^{(2)}\cdots \GG^{(r)}$
into an almost direct product of simple groups.
Then $\TT=\TT^{(1)}\cdot \TT^{(2)}\cdots \TT^{(r)}$ (an almost direct product of tori),
where each $\TT^{(\vk)}$ is a maximal torus in $\GG^{(\vk)}$ ($\vk=1,\dots, r$).
We write $\ttt=\Lie \TT$, $\ttt^{(\vk)}=\Lie \TT^{(\vk)}$,
then
\[ \ttt=\ttt^{(1)}\oplus\dots\oplus\ttt^{(r)}. \]
 The root system $R$  decomposes into a ``direct sum'' of irreducible root systems
\[
R=R^{(1)}\sqcup\dots\sqcup R^{(r)}
\]
(disjoint union), where $R^\vvk=R(\GG_\C^\vvk,\TT_\C^\vvk)$,
and we have
\[\Pi=\Pi^{(1)}\sqcup\dots\sqcup \Pi^{(r)},\]
 where each subset $\Pi^{(\vk)}$ ($\vk=1,\dots,r$) is a basis of $R^{(\vk)}$.
We have
\[
D=D^{(1)}\sqcup\dots\sqcup D^{(r)},
\]
where each connected component  $D^{(\vk)}$ ($\vk=1,\dots,r$)
is the Dynkin diagram of the irreducible root system $R^{(\vk)}$ with respect to $\Pi^{(\vk)}$.
Let $\alpha_0^\vvk\in R^\vvk$ denote the lowest root of $R^\vvk$.
Let $\Dtil^\vvk$ denote the extended Dynkin diagram of $R^\vvk$ with respect to $\Pi^\vvk$,
then the set of vertices of $\Dtil^\vvk$ is $\Pitil^\vvk:=\Pi^\vvk\cup\{\alpha_0^\vvk\}$.
We define the extended Dynkin diagram of $R$ with respect to $\Pi$ to be
\begin{equation*}
\Dtil=\Dtil^{(1)}\sqcup\dots\sqcup \Dtil^{(r)};
\end{equation*}
then the set of vertices of $\Dtil$ is
\[\Pitil=\Pitil^{(1)}\sqcup\dots\sqcup \Pitil^{(r)}=\Pi\sqcup\Pitil_0,\]
where $\Pitil_0=\{\alpha_0^{(1)},\dots,\alpha_0^{(r)}\}$.
For each $\vk=1,\dots,r$, let $(m_\beta)_{\beta\in\Pitil^\vvk}$ be the coefficients of linear dependence
\[ \sum_{\beta\in\Pitil^\vvk} m_\beta \beta=0\]
normalized so that $m_{\alpha_0^\vvk}=1$.
Then $m_\beta\in\Z, m_\beta\ge 0$ for any $\beta\in\Pitil$.

\section{Action of $P^\vee/Q^\vee$ on the extended Dynkin diagram}
\label{s:action}

First let  $\GG$ be a  {\em simple} compact $\R$-group.
Recall that $\ttt$ denotes the Lie algebra of $\TT$.
Following \cite[Section 3.3.6]{OV2}, we introduce the {\em barycentric coordinates}
$x_{\alpha_1},\dots,x_{\alpha_\ell},x_{\alpha_0}$ of a point  $x\in \ttt$ by setting
\begin{equation*}
d\alpha_j(x)=\ii x_{\alpha_j}\text{ \ for }j=1,\dots,\ell,\quad d\alpha_0(x)=\ii  (x_{\alpha_0}-1),
\end{equation*}
where $\ii^2=-1$. We write $x_j$ for $x_{\alpha_j}$.
By \eqref{eq:sum-mj} we have
\[
0=\left(\sum_{j=0}^\ell m_j\, d\alpha_j\right)(x)=\ii\left(-1+ \sum_{j=0}^\ell m_j x_j\right),
\]
hence
\begin{equation}\label{eq:app-sum-2}
\sum_{j=0}^\ell m_j x_j=1.
\end{equation}
By \cite[Section VI.2.1]{Bourbaki} and \cite[Section VI.2.2, Proposition 5(i)]{Bourbaki},
see also \cite[Section 3.3.6, Proposition 3.10(2)]{OV2},
the closed simplex $\Delta\subset\ttt$ given by the inequalities
\begin{equation*}
x_1\ge 0,\ \dots,\ x_n\ge 0,\ x_0\ge 0
\end{equation*}
is a fundamental domain for the affine Weyl group  ${Q^\vee}\rtimes W$,
where $W$ is the usual Weyl group.
This means that  every orbit of ${Q^\vee}\rtimes W$ intersects $\Delta$ in one and only one point.

Now let $\GG$ be a semisimple (not necessarily simple) compact $\R$-group.
We introduce  the barycentric coordinates $(x_\beta)_{\beta\in\Pitil}$ of $x$ defined by
\begin{equation*}
 d\beta(x)=\ii x_\beta\text{ for }\beta\in\Pi, \quad d\beta(x)=\ii (x_\beta-1) \text{ for }\beta\in\Pitil_0=\Pitil\smallsetminus\Pi,
\end{equation*}
they satisfy
\begin{equation*}
  \sum_{\beta\in\Pitil^{(\vk)}} m_\beta x_\beta=1\quad\text{for each }\vk=1,\dots,r,
\end{equation*}
see \eqref{eq:app-sum-2}.
Write $\ttt=\bigoplus_{\vk=1}^r\ttt_\vk$.
For each $\vk=1,\dots,r$,
let $\Delta^\vvk$ denote the closed simplex in $\ttt^\vvk$ given by the inequalities
\[x_\beta\ge0\quad\text{for }\beta\in\Pitil^\vvk.\]
Then the product $\Delta=\prod_{\vk=1}^r \Delta^{(\vk)}$ is the closed subset in $\ttt$ given by the inequalities
\[x_\beta\ge0\quad\text{for }\beta\in\Pitil,\]
and $\Delta$ is a fundamental domain for the affine Weyl group  $Q^\vee\rtimes W$ in $\ttt$,
acting as in Section \ref{s:notation}.
Again, this means that  every orbit of ${Q^\vee}\rtimes W$ intersects $\Delta$ in one and only one point.

The group $(X^\vee\rtimes W)/(Q^\vee\rtimes W)=X^\vee/Q^\vee\simeq \pi_1(G)$ acts on $\Delta$.
We wish to describe this action.
Since $X^\vee/Q^\vee\subset P^\vee/Q^\vee$, it suffices to describe the action of $P^\vee/Q^\vee$,
and it suffices to consider the case when $R$ is irreducible.

From now on till the end of this section we assume that $R$ is an irreducible root system.
The action of $P^\vee/Q^\vee$ on $\Delta$
is given by permutations of coordinates corresponding to a subgroup
of the automorphism group of the extended Dynkin diagram
acting simply transitively on the set of vertices $\alpha_j$ with $m_j=1$.
This action is described in \cite[Section VI.2.3, Proposition 6]{Bourbaki}.

Namely, let $\omega_1^\vee, \dots,\omega_\ell^\vee$ denote
the set  of fundamental coweights,
i.e., the basis of $P^\vee$ dual to the basis $\alpha_1,\dots,\alpha_\ell$ of $Q$.
Then the nonzero cosets of $P^\vee/Q^\vee$ are represented by the fundamental coweights $\omega_j^\vee$ such that $\ii\omega_j^\vee$   belongs to $\Delta$,
i.e., by those $\omega_j^{\vee}$ with $m_j=1$.
Let $w_0$, resp.\ $w_j$, denote the longest element in $W$, resp.\
in the Weyl group $W_j$ of the root subsystem $R_j$ generated by $\Pi\smallsetminus\{\alpha_j\}$.
Then the transformation
\begin{equation} \label{e:coset}
x\mapsto w_j w_0 x+\ii\omega_j^{\vee}
\end{equation}
preserves $\Delta$ whenever $m_j=1$ and gives the action of the respective coset $[\omega_j^\vee]\in P^\vee/Q^\vee$ on $\Delta$.

Observe that the affine transformation \eqref{e:coset} is an isometry of the Euclidean structure on $\ttt$ given by the restriction of the Killing form. Hence the action of $[\omega_j^\vee]$ preserves the Euclidean polytope structure of the simplex $\Delta$. In particular, it permutes the vertices of $\Delta$, which are equal to $v_i=\ii\omega_i^\vee/m_i$ ($i=1,\dots,\ell$) and $v_0=0$, and the facets $\Delta_i$ of $\Delta$, which correspond to the roots $\alpha_i\in\Pitil$ ($i=1,\dots,\ell,0$), preserving the angles between the facets.
Hence the action of $[\omega_j^\vee]$ induces a permutation $\sigma=\sigma_j$ of the set $\{1,\dots,\ell,0\}$ such that the facet $\Delta_i$ maps to $\Delta_{\sigma(i)}$, and the opposite vertex $v_i$ is mapped to $v_{\sigma(i)}$. In particular, $\sigma_j$ takes $0$ to $j$.

Since the relative lengths of the roots in $\Pitil$ and the angles between them and between the respective facets of $\Delta$ are read off from the extended Dynkin diagram $\Dtil$, the permutation $\sigma$ comes from an automorphism of $\Dtil$. Furthermore, the action of $[\omega_j^\vee]$ permutes the barycentric coordinates $x_i$ of a point $x\in\Delta$, because they are determined by the vertices $v_i\in\Delta$. Namely, any $x\in\Delta$ is mapped to $x'\in\Delta$ with coordinates $x'_i=x_{\sigma^{-1}(i)}$. One obtains an action of $P^\vee/Q^\vee$ on $\Dtil$, which we describe below explicitly case by case, using \cite[Planches I-IX, assertion (XII)]{Bourbaki}.

If $\GG$ is of one of the types $\EE_8,\ \FF_4, \GG_2$, then $P^\vee/Q^\vee=0$.
If $\GG$ is of one of the types $\AA_1,\ \BB_\ell\ (\ell\ge 3),\ \CC_\ell\ (\ell\ge 2),\ \EE_7$, then $P^\vee/Q^\vee\simeq\Z/2\Z$,
and the nontrivial element $P^\vee/Q^\vee$ acts on $\Dtil$ by the only nontrivial automorphism of $\Dtil$:
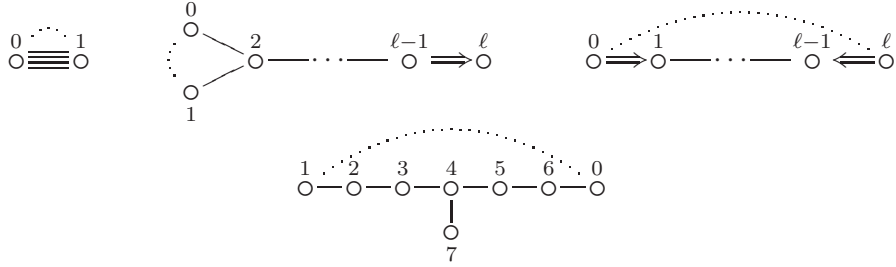
\begin{equation*}
\xymatrix@1@R=-5pt@C=15pt{
&& & \bc{0}\ar@{-}[rd]  \ar@/_0.7pc/@{.}[dd] &&\\
 \bc{0}\ar@<-2pt>@2{-}[r] \ar@<+2pt>@2{-}[r] \ar@/^1.0pc/@{.}[r] &\bc{1}
&& & \bc{2}\rline & \cdots & \lline \bc{\ell-1} \ar@{=>}[r] & \bc{\ell}&&
\bc{0}\ar@{=>}[r] \ar@/^1.5pc/@{.}[rrrr]
&\bc{1} \rline& \cdots & \lline \bc{\ell-1} & \ar@{=>}[l] \bc{\ell}\\
&& & \bcu{1} \ar@{-}[ru]}
\end{equation*}
\begin{equation*}
\mxymatrix{  \bc{1} \rline \ar@/^1.8pc/@{.}[rrrrrr]
&  \bc{2} \rline &  \bc{3} \rline &  \bc{4}
\rline \dline &  \bc{5} \rline &  \bc{6}\rline &\bc{0} \\
& & & \bcu{7} & & }
\end{equation*}

It remains to consider the cases $\AA_\ell$ $(\ell\ge 2)$, $\DD_\ell$ and $\EE_6$.
In order to describe the action of the group $P^\vee/Q^\vee$ on  $\Dtil$, it suffices to describe its action on the set of vertices $\alpha_j$ of $\Dtil$ with $m_j=1$.
These are the images of $\alpha_0$ under the automorphism group of $\Dtil$.
\medskip

Let $D$ be of type $\AA_\ell$, $\ell\ge2$.
The generator $[\omega_1^\vee]$ of $P^\vee/Q^\vee$ acts on $\Dtil$ as the
cyclic permutation $0\mapsto 1\mapsto\dots\mapsto\ell-1\mapsto\ell\mapsto 0$:
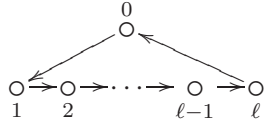
\[\sxymatrix{&&\bc{0}\ar@{->}[dll] \ar@{<-}[drr] \\\bcu{1} \ar@{->}[r] &\bcu{2}\ar@{->}[r] &\cdots\ar@{->}[r] &\bcu{\ell-1}\ar@{->}[r]& \bcu{\ell} }\]
\bigskip

Let $D$ be of type $\DD_\ell$, $\ell\ge 4$ is even.
We have $P^\vee/Q^\vee\simeq\Z/2\Z\oplus\Z/2\Z$,
and  the classes $[\omega_1^\vee]$ and $[\omega_{\ell-1}^\vee]$ are generators of    $P^\vee/Q^\vee$.
These generators act on $\Dtil$ as follows: $[\omega_1^\vee]$ acts as \ $0\leftrightarrow1$, \ \ $\ell-1\leftrightarrow\ell$, \
and $[\omega_{\ell-1}^\vee]$ acts as \ $0\leftrightarrow\ell-1$, \ \ $1\leftrightarrow\ell$:
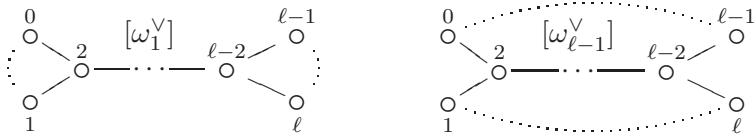
\begin{equation*}
\xymatrix@1@R=-5pt@C=10pt{
\bc{0} \ar@{-}[rd]  \ar@/_0.7pc/@{.}[dd]
& & [\omega_1^\vee] & &\bc{\ell-1}\ar@{-}[ld]  \ar@/^0.7pc/@{.}[dd]\\
&\bc{2}\rline & \cdots  \rline & \bc{\ell-2}\\
\bcu{1}\ar@{-}[ru] & &  & &\bcu{\ell}\ar@{-}[lu] }
\qquad\qquad
\xymatrix@1@R=-5pt@C=10pt{
\bc{0} \ar@{-}[rd]\ar@/^1pc/@{.}[rrrr]
&&[\omega_{\ell-1}^\vee] &&\bc{\ell-1}\ar@{-}[ld] \\
&\bc{2}\rline & \cdots  \rline & \bc{\ell-2}\\
\bcu{1}\ar@{-}[ru] \ar@/_0.9pc/@{.}[rrrr]
 & &  & &\bcu{\ell}\ar@{-}[lu] }
\end{equation*}
\medskip

Let $D$ be of type $\DD_\ell$, $\ell\ge 5$ is odd.
We have $P^\vee/Q^\vee\simeq\Z/4\Z$, and the class $[\omega_{\ell-1}^\vee]$
is a generator of $P^\vee/Q^\vee$.
This generator acts on $\Dtil$ as  the 4-cycle $0\mapsto\ell-1\mapsto1\mapsto\ell\mapsto0$:
 
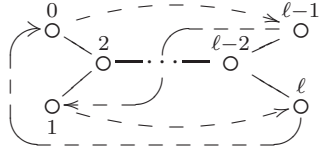
\[   \xymatrix@1@R=-5pt@C=10pt{
&& \bc{0} \ar@{-}[rd]\ar@/^0.75pc/@{-->}[rrrr] &&  &&\bc{\ell-1}\ar@{-}[ld] \ar@{-->} `l[lld]  `[lldd]  [lllldd] \\
&&&\bc{2}\rline & \cdots  \rline & \bc{\ell-2} \\
&&\bcu{1}\ar@{-}[ru] \ar@/_0.7pc/@{-->}[rrrr] & &  & &\bc{\ell}\ar@{-}[lu] \ar@{--2>} `d[l] `l[lllll] `[uu]  [lllluu] \\
}   \]
 \bigskip

Let $D$ be of type $\EE_6$.
The generator $[\omega_1^{\vee}]\in P^\vee/ Q^{\vee}$  acts as the 3-cycle $0\mapsto1\mapsto5\mapsto0$:
\medskip

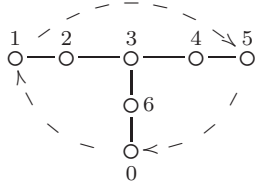
\begin{equation*}
\xymatrix@1@R=5pt@C=10pt{ \bc{1} \rline \ar@/^1.7pc/@{-->}[rrrr]
& \bc{2} \rline & \bc{3} \rline \dline & \bc{4} \rline & \bc{5} \ar@/^1pc/@{-->}[lldd] \\
& &  {\phantom{\hbox{\SMALL\hs 6}}}\ccc \hbox{\SMALL\hs 6}\dline & & \\
&&\bcu{0} \ar@/^1pc/@{-->}[lluu]
}
\end{equation*}

\section{$n$-th roots of a central element}
\label{s:roots}

Let $\GG$ a  compact semisimple $\R$-group, not necessarily simple. Let $\TT$, $G$, $T$, $X$, $D$, $\Dtil$, ets. be as in Section \ref{s:notation}.

Let $z\in Z_G$ and  let $n$ be a positive integer. We consider the set of $n$-th roots of $z$ in $G$
\[
G_n^z:=\{g\in G\ |\ g^n=z\}.
\]
In particular, $G_n:=G_n^1$ is the set of $n$-th roots of 1 in $G$, i.e., the set of elements of order dividing $n$ in $G$.

The group $G$ acts on $G_n^z$ on the left by conjugation  $g * a=g a g^{-1}$ ($g\in G,\ a\in G_n^z$).
We wish to compute the set $G_n^z/\!\!\!\sim$ of $n$-th roots of $z$ modulo conjugation.

Consider the set $T_n^z\subset G_n^z$ (note that $z\in Z\subset T$). The group $W$ acts on $T_n^z$ on the left by
\begin{equation}\label{e:ast}
w * t=ntn^{-1},
\end{equation}
where $w=nT\in W,\ n\in N,\ t\in T$.
It is easy to see  that the embedding  $T_n^z\into G_n^z$ induces a bijection $T_n^z/W\isoto G_n^z/\!\!\!\sim$.
Thus we wish to compute  $T_n^z/W$.

We describe the set $T_n^z/W$ in terms of {\em Kac $n$-labelings of $\Dtil$}.

\begin{definition}
A {\em Kac $n$-labeling} of an extended Dynkin diagram $\Dtil=\Dtil^{(1)}\sqcup\dots\sqcup \Dtil^{(r)}$,
where each $\Dtil^{(\vk)}$ is connected for $\vk=1,\dots,r$,
is a family of nonnegative integer  numerical labels $\pp=(p_\beta)_{\beta\in \Pitil}\in \Z_{\ge 0}^\Pitil$
at the vertices $\beta\in\Pitil$  of $\Dtil$ satisfying
\begin{equation}\label{e:vk-p-0}
  \sum_{\beta\in\Pitil^{(\vk)}} m_\beta \hs p_\beta=n\quad\text{for each }\vk=1,\dots,r.
  \end{equation}
\end{definition}

Note that a Kac $n$-labeling $\pp$ of $\Dtil=\Dtil^{(1)}\sqcup\dots\sqcup \Dtil^{(r)}$ is the same as
a family $(\pp^{(1)},\dots \pp^{(r)})$, where each $\pp^\vvk$ is a Kac $n$-labeling of $\Dtil^\vvk$.

 Let $z\in Z_G\subset T$.
We write
\begin{equation}\label{e:l-z}
{z}=\exp\,2\pi\ii \zeta,\quad \text{where }\ \zeta\in\ttt_\C.
\end{equation}
For $\lambda\in X$ consider $d\lambda(\zeta)\in \C$. We have
\begin{equation}\label{e:z}
\exp 2\pi\ii\, d\lambda(\zeta)=\exp d\lambda(2\pi\ii \zeta)
=\lambda(\exp 2\pi\ii \zeta)=\lambda(z).
\end{equation}
Since $z$ is an element of finite order in $T$, we see that
$\lambda(z)$ is a root of unity, hence  by \eqref{e:z} $d\lambda(\zeta)\in \Q$,
and it follows from \eqref{e:z} that the image of $d\lambda(\zeta)$ in $\Q/\Z$ depends only on $z$,
and  not on the choice of $\zeta$.
Note that if $\lambda\in Q\subset X$, then $\lambda(z)=1$, hence $d\lambda(\zeta)\in\Z$.

\begin{notation}\label{n:sK}
We denote by $\sK_n$ the set of Kac $n$-labelings of $\Dtil$,
i.e., the set of $\pp=(p_\beta)\in\Z_{\ge0}^\Pitil$ satisfying \eqref{e:vk-p-0}.
We denote by $\sK_{n,\R}$  the set of families $\pp=(p_\beta)\in \R_{\ge0}^\Pitil$ satisfying \eqref{e:vk-p-0}, i.e., the set of tuples of barycentric coordinates of points in $n\Delta$.
For $z\in Z_G$, we denote by $\sK_n^z$ the set of Kac $n$-labelings $\pp\in\sK_n$  of $\Dtil$ satisfying
\begin{align}\label{eq:cond-z}
&\sum_{\alpha\in\Pi}c_\alpha p_\alpha\equiv d\lambda(\zeta)\pmod{\Z} \\
&\text{for any generator } [\lambda] \text{ of } X/Q   \text{ with \ }
\lambda=\sum_{\alpha\in\Pi}c_\alpha \alpha\, ,\nonumber
\end{align}
where $\zeta$ is as in \eqref{e:l-z}.
Condition \eqref{eq:cond-z} does not depend on the choice of $\zeta$ satisfying \eqref{e:l-z}.
We have $\sK_n^z\subset \sK_n\subset \sK_{n,\R}$.
The group $X^\vee/Q^\vee$ acts on $\sK_{n,\R}$ and $\sK_n$ via the action on $\Dtil$.
We shall see below that the subset $\sK_n^z$ of $\sK_n$ is $X^\vee/Q^\vee$-invariant.
\end{notation}

\begin{construction}\label{constr:varphi}
Let $\pp=(p_\beta)\in\sK_{n,\R}$.
Set
\[\xx=(x_\beta)_{\beta\in\Pitil}:=(p_\beta/n)_{\beta\in\Pitil}\in\sK_{1,\R},\]
then there exists a point $x\in\Delta\subset \ttt$ with barycentric coordinates $(x_\beta)_{\beta\in\Pitil}$.
We set
\[ \varphi(\pp)=e(x):=\exp\hs 2\pi x\in T. \]
\end{construction}

The following theorem gives a combinatorial description of the set $T^z_n/W$ in terms of Kac $n$-labelings.
It generalizes a result of Kac \cite{Kac}, who described, in particular, the set $T_n/W$ in the case when $\GG$ is an adjoint group.

\begin{theorem}\label{thm:p-z}
Let $\GG$ be a compact semisimple $\R$-group,
$\TT\subset \GG$ be a maximal torus, $R=R(\GG_\C,\TT_C)$ be the corresponding root system,
$\Pi$ be a basis of $R$,
$\Dtil=\Dtil(\GG,\TT,\Pi)$ be the corresponding extended Dynkin diagram.
Let $n$ be a positive integer.
Let $z\in Z_G$ be a central element.
Then the subset $\sK_n^z\subset\sK_n$ is $X^\vee/Q^\vee$-invariant, and
the map $\varphi\colon \sK_{n,\R}\to T$ of Construction \ref{constr:varphi} induces a bijection
\begin{equation}\label{e:varphi-ast}
\varphi_*\colon \sK_n^z/(X^\vee/Q^\vee)\isoto T_n^z/W
\end{equation}
between the set of $X^\vee/Q^\vee$-orbits in $\sK_n^z$
and  the set  of $W$-orbits in $T_n^z$.
\end{theorem}

\begin{proof}
Consider a $W$-orbit $[a]$ in $T/W$, where $a\in T$.
Write
$a=e(x)$ for some $x\in \ttt$.
The map $e\colon \ttt\to T$ is $W$-equivariant.
The group $X^\vee$ acts on the set $\ttt$ by translations, and the map $e$ induces a bijection $\ttt/X^\vee\isoto T$,
hence it induces a bijection
\[ \ttt/(X^\vee\rtimes W)\isoto T/W.\]
Since $\Delta$ is a fundamental domain of the normal subgroup $Q^\vee\rtimes W\subset X^\vee\rtimes W$  (see Section \ref{s:action}),
after changing the representative $a\in T$ of $[a]\in T/W$ we may choose $x$  lying in $\Delta$,
and such $x$ is unique up to the action of the quotient group $(X^\vee\rtimes W)\hs/\hs(Q^\vee\rtimes W) = X^\vee/Q^\vee$.
We see that the map $e$ induces a bijection
\[ \Delta/(X^\vee/Q^\vee)\isoto T/W.\]
The map
\begin{equation}\label{e:pp-xx-x}
 \sK_{n,\R}\to\Delta, \quad \pp\mapsto \ \xx=\pp/n \ \mapsto x
\end{equation}
is a $P^\vee/Q^\vee$-equivariant bijection,   hence it induces a  bijection
\[   \sK_{n,\R}/(X^\vee/Q^\vee) \isoto\Delta/ (X^\vee/Q^\vee).\]
We see that the map $\varphi \colon \sK_{n,\R}\to T$ induces a bijection
\begin{equation}  \label{e:bijection-over-R}
\sK_{n,\R}/(X^\vee/Q^\vee)\isoto  T/W.
\end{equation}
In particular, two tuples $\pp,\pp'\in\sK_{n,\R}$ are in the same  $X^\vee/Q^\vee$-orbit if and only if $\varphi(\pp),\varphi(\pp')\in T$ are in the same $W$-orbit.

Now we wish to describe  $\pp=(p_\beta)\in\sK_{n,\R}$ such that $\varphi(\pp)\in T_n^z$, i.e., $\varphi(\pp)^n=z$.
For $x\in\Delta$ obtained from $\pp\in\sK_{n,\R}$ as in \eqref{e:pp-xx-x}, the assertion that  $e(x)^n=z$ is equivalent to the condition
$$
\lambda(\exp 2\pi n x)=\lambda(\exp 2\pi\ii\hs \zeta)
$$
for all $\lambda\in X$,
which in turn is equivalent to
$$
-\ii n\hs d\lambda(x)\equiv d\lambda(\zeta)\pmod{\Z}.
$$
We write $\lambda=\sum_{{\alpha}\in\Pi}c_{\alpha}\alpha$ and obtain
$$
-\ii n\sum_{{\alpha}\in\Pi}c_{\alpha}\, d\alpha(x)\equiv
d\lambda(\zeta) \pmod{\Z}.
$$
Since $d\alpha(x)=\ii x_{\alpha}$ for $\alpha\in\Pi$, and $nx_\alpha=p_\alpha$, we obtain
\[
\sum_{\alpha\in \Pi}c_\alpha p_\alpha=n\sum_{{\alpha}\in\Pi}c_{\alpha} x_{\alpha}\equiv d\lambda(\zeta) \pmod{\Z}.
\]
Thus $\varphi(\pp)\in T_n^z$ if and only if
\begin{equation}\label{eq:cond-gen}
\sum_{\alpha\in \Pi}c_\alpha p_\alpha\equiv d\lambda(\zeta) \pmod{\Z} \text{ \  for any \ }\lambda\in X
\text{ with \ }\lambda=\sum_{\alpha\in \Pi} c_\alpha\alpha.
\end{equation}

Assume that $\varphi(\pp)\in T_n^z$, then \eqref{eq:cond-gen} holds.
Observe that for $\lambda=\alpha\in\Pi$, condition \eqref{eq:cond-gen} means that $p_\alpha\in\Z$, because $d\alpha(\zeta)\in\Z$.
Since $p_\alpha\in\Z$ for all $\alpha\in\Pi$,  by \eqref{e:vk-p-0} we have $p_\beta\in\Z$  for any $\beta\in\Pitil_0=\Pitil\smallsetminus\Pi$,
because $m_{\beta}=1$. Thus $\pp\in\sK_n$\hs.
Condition \eqref{eq:cond-z} is a special case of \eqref{eq:cond-gen}.
We conclude that $\pp\in \sK_n^z$\hs.

Conversely, assume that  $\pp\in \sK_n^z\subset \sK_n$\hs, then
condition \eqref{eq:cond-gen} holds for $\lambda=\alpha$ for any $\alpha\in\Pi$.
Since condition \eqref{eq:cond-gen} is additive in $\lambda$
(i.e., it holds for any integer linear combination of two weights $\lambda,\lambda'\in P$ whenever it holds for $\lambda$ and $\lambda'$),
it holds for any $\lambda\in Q$, because $\Pi$ generates $Q$ as an abelian group.
Now condition \eqref{eq:cond-z} implies that \eqref{eq:cond-gen} holds for for all $\lambda\in X$.
We conclude that $\varphi(\pp)\in T_n^z$\hs.

Thus $\varphi(\pp)\in T_n^z$ if and only if $\pp\in \sK_n^z$\hs.
Since the subset $T_n^z\subset T$ is $W$-invariant, we conclude that the subset $\sK_n^z\subset K_{n,\R}$ is $X^\vee/Q^\vee$-invariant.
Bijection \eqref{e:bijection-over-R} induces  \eqref{e:varphi-ast}, which proves the theorem.
\end{proof}

We need another version of Theorem \ref{thm:p-z}.
We start from a Kac $n$-labeling $\qq=(q_\beta)\in \sK_n$ of $\Dtil$.
Set $z=\varphi(\qq)^n$.
It follows from the proof of Theorem \ref{thm:p-z} that $z\in Z_G$.

\begin{corollary}\label{cor:p-q}
With the assumptions and notation of Theorem \ref{thm:p-z}, let $\qq$ be an $n$-labeling of $\Dtil$.
Set $z=\varphi(\qq)^n\in Z_G$.
Then the subset $\sK_n^{(\qq)}\subset\sK_n$ consisting of  Kac $n$-labelings $\pp\in\sK_n$ of $\Dtil$ satisfying
\begin{align}\label{eq:cond-q}
&\sum_{\alpha\in\Pi}c_\alpha p_\alpha\equiv \sum_{\alpha\in\Pi}c_\alpha q_\alpha\pmod{\Z} \\
&\text{for any generator } [\lambda] \text{ of } X/Q   \text{ with \ }
\lambda=\sum_{\alpha\in\Pi}c_\alpha \alpha\,,\nonumber
\end{align}
is $X^\vee/Q^\vee$-invariant, and the map $\varphi$ of Construction \ref{constr:varphi}
induces a bijection between  $\sK_n^{(\qq)}/(X^\vee/Q^\vee)$ and  $T_n^z/W$.
\end{corollary}

Indeed, by Theorem \ref{thm:p-z} we have $\qq\in \sK_n^z$, hence $\sK_n^{(\qq)}=\sK_n^z$, and the corollary follows from the theorem.

\section{Real Galois cohomology}\label{s:cohomology}

We denote by $H^1(\R,\HH)$  the first (nonabelian) Galois cohomology set of an $\R$-group $\HH$.
By definition, $H^1(\R,\HH)=Z^1(\R,\HH)/\!\!\sim$,
where $Z^1(\R,\HH)={\{c\in \HH(\C)\ |\ c\bar{c}=1\}}$, and $c\sim c'$ if there exists $h\in \HH(\C)$ such that $c'=h^{-1}c\hs\overline{h}$.
We say that $c\in Z^1(\R,\HH)$ is a {\em cocycle.}

Let $\HH(\R)_2\subset\HH(\R)$ denote the subset of elements of order dividing 2. If $b\in \HH(\R)_2$,  then
\[  b\overline{b} =b^2=1,\]
hence $b$ is a cocycle. Thus $\HH(\R)_2\subset Z^1(\R,\HH)$.

Let $\GG$ be a connected, compact (anisotropic), semisimple algebraic group over the field of real numbers $\R$.
Let $\TT\subset\GG$ be a maximal torus.
We use the notation of Section \ref{s:notation}.

\begin{theorem} \label{thm:Kac}
Let $\GG$ be a connected, compact, semisimple algebraic $\R$-group.
There is a canonical bijection between  the set of $P^\vee/Q^\vee$-orbits
in the set $\sK_2$ of Kac $2$-labelings of the extended Dynkin diagram $\Dtil=\Dtil(\GG,\TT,\Pi)$
and the first Galois cohomology set $H^1(\R,\GG^\ad)$.
\end{theorem}

We specify the bijection.
Consider the map $\varphi^\ad\colon \sK_{2,\R}\to T^\ad$ of Construction \ref{constr:varphi} for $\GG^\ad$,
it sends $\sK_2\subset \sK_{2,\R}$ to $(T^\ad)_2$, where $(T^\ad)_2$ denotes the set of elements of order dividing 2 in $T^\ad$.
The bijection of the theorem sends the $P^\vee/Q^\vee$-orbit of $\pp\in \sK_2$ to the cohomology class $[\varphi(\pp)]\in H^1(\R,\GG^\ad)$
of $\varphi(\pp)\in (T^\ad)_2\subset Z^1(\R,\GG^\ad)$.

This result goes back to Kac \cite{Kac}.
In the last sentence of \cite{Kac} Kac notes that his results yield a classification of real forms of simple Lie algebras.
Inner real forms of a compact simple group $\GG$ (or of its Lie algebra $\Lie\,\GG$)
are classified by the orbits of the group $\Aut\,\Dtil=(P^\vee/Q^\vee)\rtimes \Aut\, D$
in the set $\sK_2$ of Kac 2-labelings  of $\Dtil$.
Those orbits and the corresponding real forms  are listed in \cite[Table 7, Types I and II]{OV}.

\begin{proof}
By Theorem \ref{thm:p-z} for the adjoint group $\GG^\ad$, the map $\varphi^\ad$ induces a bijection $\sK_2/(P^\vee/Q^\vee)\isoto (T^\ad)_2/W$.
By \cite[Section III.4.5, Example (a)]{Serre} the map sending an element
$t^\ad\in (T^\ad)_2\subset Z^1(\R,\GG^\ad)$ to its cohomology class  $[t^\ad]\in H^1(\R, \GG^\ad)$
induces a bijection $(T^\ad)_2/W\isoto H^1(\R,\GG^\ad)$, and the theorem follows.
\end{proof}

Let $_c\GG$ be an inner twisted form of a compact semisimple $\R$-group $\GG$,
where $c\in Z^1(\R, \GG^\ad)$.
By Theorem \ref{thm:Kac} the cocycle $c$ is equivalent to a cocycle of the form $t^\ad=\varphi^\ad(\qq)\in (T^\ad)_2\subset Z^1(\R, \GG^\ad)$
for some Kac 2-labeling $\qq=(q_\beta)_{\beta\in\Pitil}$ of $\Dtil$.
We have $t^\ad=\exp 2\pi y$, where $y\in\Delta$ has barycentric coordinates $y_\beta=q_\beta/2$ for $\beta\in\Pitil$.
It follows that $t^\ad$ is determined by the equations
\[
\alpha(t^\ad)=(-1)^{q_\alpha}\quad \text{for } \alpha\in\Pi.
\]
We can twist $\GG$ using $t^\ad$; we denote the obtained twisted form by $_\qq \GG$, then $_c \GG\simeq\hs_\qq \GG$.
Note that there is a canonical isomorphism between $\TT$ and the twisted torus $_\qq \TT$,
because the inner automorphism of $\GG$ defined by $t^\ad$ acts on $\TT$ trivially.
It follows that $\TT$ canonically embeds into $_\qq \GG$, in particular, $T_2\subset \hs_\qq\GG(\R)_2\subset Z^1(\R,\hs_\qq \GG)$.

We compute $H^1(\R,\hs_\qq\GG)$.
Set $t=\varphi(\qq)\in T$, where  $\varphi\colon\sK_{2,\R}\to T$ is the map of Construction \ref{constr:varphi}.
Then the image of $t$ in $T^\ad$ is $t^\ad$.
Since $(t^\ad)^2=1$, we see that $t^2\in Z_G$.
Set $z=t^2$, then $t\in T_2^z$.

\begin{lemma}\label{prop:H-T2z}
There is a bijection $T_2^z/W\isoto H^1(\R,\hs _\qq \GG)$ that sends the $W$-orbit of $a\in T_2^z$
to the cohomology class of $at^{-1}\in T_2\subset Z^1(\R,\hs_\qq\GG)$.
\end{lemma}

\begin{proof}
Recall that we have the standard left action $\ast$ of $W$ on $T_2^z$ given by formula \eqref{e:ast}.
We define the $t^\ad$-twisted left action $\ast_{t^\ad}$ of $W$ on $T_2$ as follows:
let  $w=nT\in W$, $n\in N$, $b\in T_2$, then
\[ w\ast_{t^\ad}\hs b= n b\hs t n^{-1} t^{-1}.\]

We define a bijection
\begin{equation}\label{e:twisted-standard}
 a\mapsto at^{-1}\colon T_2^z\to T_2
\end{equation}
(which takes $t$ to $1$).
We have
\[ (w\ast a)\hs t^{-1}=n a  n^{-1} t^{-1}=n (at^{-1})\hs t n^{-1} t^{-1}=w\ast_{t^\ad}(at^{-1}), \]
hence, the standard left action $\ast$ of $W$ on $T_2^z$ is compatible with the  $t^\ad$-twisted left action $\ast_{t^\ad}$ of $W$ on $T_2$
with respect to bijection \eqref{e:twisted-standard}.
We obtain a bijection $T_2^z/W=T_2^z/\ast W\isoto T_2/\ast_{t^\ad}W$ between the sets of $W$-orbits.

By \cite[Theorem 1]{Bo}, see also \cite[Theorem 9]{Borovoi-arXiv},
the map sending $b\in T_2\subset Z^1(\R,\hs_\qq\GG)$ to its cohomology class $[b]\in H^1(\R,\hs_\qq\GG)$
induces a bijection $T_2/\ast_{t^\ad}W\isoto H^1(\R,\hs_\qq\GG)$.

Combining these two bijections, we obtain the bijection of the lemma.
\end{proof}

The following theorem is the main result of this paper.
It gives a combinatorial description of the first Galois cohomology set $H^1(\R,\hs_\qq \GG)$ of an inner twisted form $_\qq \GG$
of a compact semisimple $\R$-group $\GG$ in terms of Kac 2-labelings of the extended Dynkin diagram of $\GG$.

\begin{theorem}\label{t:main}
Let $\GG$ be a connected, compact, semisimple algebraic $\R$-group.
Let $\TT\subset \GG$ be a maximal torus and $\Pi$ be a basis of the root system $R=R(\GG_\C,\TT_\C)$.
Let $\qq$ be a Kac 2-labeling of the extended Dynkin diagram $\Dtil=\Dtil(\GG,\TT,\Pi)$.
Then the subset $\sK_2^{(\qq)}\subset \sK_2$ of Kac 2-labelings $\pp$ of $\Dtil$
satisfying  condition \eqref{eq:cond-q} of Corollary \ref{cor:p-q}
is $X^\vee/Q^\vee$-invariant, and there is a bijection  between the set of orbits $\sK_2^{(\qq)}/(X^\vee/Q^\vee)$
and  the first Galois cohomology  set $H^1(\R,\hs_\qq \GG)$.
\end{theorem}

We specify the bijection of the theorem. It is induced by the map sending
a  Kac 2-labeling $\pp\in\sK_2$  satisfying \eqref{eq:cond-q}
to the cocycle  $\exp 2\pi u\in T_2\subset Z^1(\R,\hs_\qq \GG)$,
where $u\in\ttt$ is the the element with barycentric coordinates $u_\alpha=(p_\alpha-q_\alpha)/2$ for $\alpha\in\Pi$.
In particular, this bijection sends the $X^\vee/Q^\vee$-orbit of $\qq$
to the neutral element of $H^1(\R,\hs_\qq \GG)$.

\begin{proof}[Proof of Theorem \ref{t:main}]
By Corollary \ref{cor:p-q} there is a bijection between
the set of orbits of $X^\vee/Q^\vee$ in the set of Kac 2-labelings $\pp\in\sK_2$ of $\Dtil$ satisfying \eqref{eq:cond-q} and the set $T_2^z/W$,
which sends the $X^\vee/Q^\vee$-orbit of $\pp$
to the $W$-orbit of $\exp 2\pi x\in T_2^z$,
where $x\in\ttt$ is the element with barycentric coordinates $x_\beta=p_\beta/2$ for $\beta\in\Pitil$.
By Lemma \ref{prop:H-T2z} there is a bijection $ T_2^z/W\isoto H^1(\R,\hs _\qq \GG)$,
which sends the $W$-orbit of an element $a\in T_2^z$ to the cohomology class of $at^{-1}\in T_2\subset Z^1(\R,\hs _\qq \GG)$.
We compose these two bijections.
Since $t=\exp 2\pi y$, where $y\in\ttt$ is the element with barycentric coordinates $y_\beta=q_\beta/2$ for $\beta\in\Pitil$,
the composite  bijection sends the $X^\vee/Q^\vee$-orbit of a Kac 2-labeling $\pp$ satisfying \eqref{eq:cond-q}
to the cohomology class of
\[ \exp 2\pi x\cdot (\exp 2\pi y)^{-1}=\exp 2\pi (x-y)=\exp 2\pi u \in T_2\subset Z^1(\R,\hs_\qq \GG),\]
where $u:=x-y\in\ttt$ has barycentric coordinates $u_\alpha=(p_\alpha-q_\alpha)/2$ for $\alpha\in\Pi$.
Clearly this  composite  bijection  sends $\pp=\qq$ to the cohomology class of $1\in Z^1(\R,\hs_\qq \GG)$,
thus to the neutral element of $H^1(\R,\hs_\qq \GG)$ .
\end{proof}

\section{Example: forms of $\EE_7$}
\label{s:E7}

Let $\GG$ be the simply connected compact group $\GG$ of type $\EE_7$.
Since $\GG$ is simply connected, we have $X=P$.

Below in the left hand side we give the extended Dynkin diagram $\Dtil$ of $\GG_\C$
with the numbering of vertices of \cite[Table 1]{OV},
and in the right hand side we give $\Dtil$  with the coefficients $m_j$ from \cite[Table 6]{OV}, see \eqref{eq:sum-mj}.
We have $X/Q=P/Q\simeq \Z/2\Z$, and there is $\lambda\in X\smallsetminus Q$
with
\begin{equation}\label{e:lambda-E7}
\lambda=\half(\alpha_1+\alpha_3+\alpha_7),
\end{equation}
 see e.g.~\cite[Table 3]{OV}.
In the left-hand side diagram below we mark in black the roots appearing
(with  non-integer  half-integer coefficients) in formula \eqref{e:lambda-E7}:
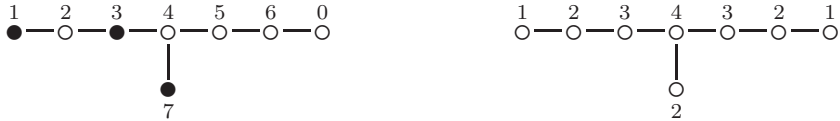
\begin{equation*}
\sxymatrix{  \bcb{1} \rline &  \bc{2} \rline &  \bcb{3} \rline &  \bc{4}
\rline \dline &  \bc{5} \rline &  \bc{6}\rline &\bc{0} \\
& & & \bcbu{7} & & }
\qquad\qquad\qquad
\sxymatrix{  \bc{1} \rline &  \bc{2} \rline &  \bc{3} \rline &  \bc{4}
\rline \dline &  \bc{3} \rline &  \bc{2} \rline &  \bc{1} \\
& & & \bcu{2} & & }
\end{equation*}
The Kac 2-labelings of $\Dtil$ are:
\begin{align*}
&\qq^{(1)}=000\two{0}{0} 002\qquad \qq^{(2)}=200\two{0}{0} 000\\
&\qq^{(3)}=100\two{0}{0}001\\
&\qq^{(4)}=010\two{0}{0} 000\qquad \qq^{(5)}=000\two{0}{0} 010\\
&\qq^{(6)}=000\two{0}{1} 000\,.
\end{align*}

The real forms of $\EE_7$ correspond to elements of $H^1(\R, \GG^\ad)$, and by Theorem \ref{thm:Kac}
to the orbits of $P^\vee/Q^\vee$ in the set  $\sK_2$ of Kac 2-labelings of $\Dtil$.
These orbits  are:
$$\{\qq^{(1)}, \qq^{(2)}\},\quad \{\qq^{(3)}\},\quad \{\qq^{(4)}, \qq^{(5)}\},\quad \{\qq^{(6)}\},$$
hence $\# H^1(\R,\GG^\ad)=4$.

Concerning $H^1(\R,\hs_\qq\GG)$, condition \eqref{eq:cond-q} defining $\sK_2^{(\qq)}$ reads
\begin{equation*}
\half (p_1+ p_3+ p_7)\equiv \half( q_1+ q_3+q_7) \pmod{\Z},
\end{equation*}
which is equivalent to
\begin{equation*}
p_1+p_3+p_7\equiv q_1+q_3+q_7 \pmod{2}.
\end{equation*}
We say that a 2-labeling $\pp\in\sK_2$ is {\em even} (resp., {\em odd}) if the sum over the black vertices
\[ p_1+p_3+p_7\]
is even (resp., odd).
Then $\sK_2^{(\qq)}$ is the set of labelings $\pp\in \sK_2$ of the same parity as $\qq$.
Since $\GG$ is simply connected, we have $X^\vee=Q^\vee$, and by Theorem \ref{t:main} the first Galois cohomology set $H^1(\R,\,_\qq \GG)$
is in a bijection with the set $\sK_2^{(\qq)}$.

For $_\qq \GG=\EE_7$ (the compact form) we take $\qq=\qq^{(1)}$,
then $q_1+q_3+q_7=0$, hence $\qq$ is even.
For $_\qq \GG={\bf EVI}$ we take $\qq=\qq^{(4)}$, see \cite[Table 7]{OV}.
We have $q_1+q_3+q_7=0$, so again $\qq$ is even.
We see that in both cases the set  $\sK_2^{(\qq)}$ is the set of all {\em even} 2-labelings of $\Dtil$:
\begin{equation}\label{e:E7-even}
000\two{0}{0} 002\qquad 200\two{0}{0} 000\qquad
010\two{0}{0} 000\qquad 000\two{0}{0} 010\,.
\end{equation}
The set $H^1(\R,\,_\qq \GG)$ is in a bijection with the set \eqref{e:E7-even}.
In particular, $\#H^1(\R,\,_\qq \GG)=4$ in both the compact case and ${\bf EVI}$.

For $_\qq \GG={\bf EV}$ (the split form) we take $\qq=\qq^{(6)}$, see \cite[Table 7]{OV}.
We have  $q_1+q_3+q_7=1$, hence $\qq$ is odd.
For $_\qq \GG={\bf EVII}$ (the Hermitian form) we take $\qq=\qq^{(3)}$, see \cite[Table 7]{OV}.
Again we have  $q_1+q_3+q_7=1$, and again $\qq$ is odd.
In both cases the set  $\sK_2^{(\qq)}$ is the set of all {\em odd} 2-labelings of $\Dtil$:
\begin{equation}\label{e:E7-odd}
 100\two{0}{0}001\qquad    000\two{0}{1} 000\,.
 \end{equation}
The set $H^1(\R,\,_\qq \GG)$ is in a bijection with the set  \eqref{e:E7-odd}.
In particular,  $\#H^1(\R,\,_\qq \GG)=2$ in both cases ${\bf EV}$ and ${\bf EVII}$.

In each case the element $\qq\in\sK_2^{(q)}$ corresponds to the neutral element of $H^1(\R,\hs_\qq \GG)$.

\section{Example: half-spin groups}
\label{s:half-spin}

Let $\GG$ be the compact group of type $\DD_\ell$ with even $\ell=2k\ge 4$ with the cocharacter lattice
\[ X^\vee =\langle Q^\vee,\omega_{\ell-1}^\vee\rangle. \]
This compact group is neither simply connected nor adjoint, and it is  isomorphic to $\SO_{2\ell}$ only if  $\ell=4$.
It is called a half-spin group.

We show that the character lattice $X$  is generated by $Q$ and the weight
\begin{equation}\label{e:X-alphas}
 \lambda:=(\alpha_1+\alpha_3+\dots+\alpha_{\ell-3}+\alpha_\ell)/2.
\end{equation}
Indeed,
$\lambda$ is orthogonal to $\omega_{\ell-1}^\vee$ and
$\langle\lambda,\alpha^\vee\rangle=0,1,-1\in\Z$ for any $\alpha\in\Pi$.
We see that  $\lambda\in X$.
Since $\lambda\notin Q$ and $[X:Q]=2$, we conclude that $X=\langle Q,\lambda\rangle$.

Below in the left hand side we give the extended Dynkin diagram $\Dtil$ of $\GG_\C$
 with the numbering of vertices of \cite[Table 1]{OV}
(which coincides with the labeling of Bourbaki \cite{Bourbaki}).
We mark in black the roots that appear (with non-integer half-integer coefficients)
in the formula \eqref{e:X-alphas} for $\lambda$.
In the right hand side we give $\Dtil$  with the coefficients $m_j$ from \cite[Table 6]{OV}, see \eqref{eq:sum-mj}:
\[
\xymatrix@1@R=-5pt@C=10pt{
\bc{0} \ar@{-}[rd]
& & && & &\bc{\ell-1}\ar@{-}[ld]\\
&\bc{2}\rline & \bcb{3}  \rline &\cdots\rline  &\bcb{\ell-3}\rline &\bc{\ell-2}\\
\bcbu{1}\ar@{-}[ru] & & && &  &\bcbu{\ell}\ar@{-}[lu] }
\qquad\qquad
\xymatrix@1@R=-5pt@C=10pt{
\bc{1} \ar@{-}[rd] & & && & &\bc{1}\ar@{-}[ld]\\
&\bc{2}\rline & \bc{2}  \rline &\cdots\rline  &\bc{2}\rline &\bc{2}\\
\bcu{1}\ar@{-}[ru] & & && &  &\bcu{1}\ar@{-}[lu] }
 \]

Let $\pp$ be a Kac 2-labeling of the extended Dynkin diagram $\Dtil$.
We say that $\pp$ is {\em even} (resp., {\em odd}), if the sum over the black vertices
\[p_1+p_3+\dots+p_{\ell-3}+p_\ell\]
is even (resp., odd).
If $\qq\in\sK_2$ is a Kac 2-labeling of $\Dtil$, then $K_2^{(\qq)}$ is the set of Kac 2-labelings $\pp$ of the same parity as $\qq$.

The group $X^\vee/Q^\vee=\{\hs 0,\hs[\omega_{\ell-1}^\vee]\hs\}$ acts on $\Dtil$
and on the set $\sK_2$ of Kac 2-labelings of $\Dtil$.
The nontrivial element $\sigma:=[\omega_{\ell-1}^\vee]\in X^\vee/Q^\vee$
acts as the reflection with respect to the vertical  axis of symmetry of $\Dtil$, see Section \ref{s:action}, and clearly preserves the parity of labelings.
We say that a $\sigma$-orbit in $\sK_2$ is even (resp., odd),
if it consists of even (resp., odd) 2-labelings.

Let $\qq$ be a 2-labeling of $\Dtil$.
By Theorem \ref{t:main} the cohomology set $H^1(\R,\hs_\qq G)$ is in a  bijection with the set $\sK_2^{(\qq)}/(X^\vee/Q^\vee)$,
i.e., with the set of $\sigma$-orbits in $\sK_2$ of the same parity as $\qq$.
Thus in order to compute  $H^1(\R,\hs_\qq G)$ for all 2-labelings $\qq$ of $\Dtil$, it suffices to compute
the sets $\Orb^\even(\DD_\ell)$ and $\Orb^\odd(\DD_\ell)$ of the even and odd $\sigma$-orbits, respectively.
We compute also the cardinalities
\[ h^\even(\DD_\ell)=\, \#\Orb^\even(\DD_\ell)\quad\text{ and }\quad h^\odd(\DD_\ell)=\, \#\Orb^\odd(\DD_\ell).\]

We compute $\Orb^\even(\DD_\ell)$. Recall that $\ell=2k$.
For representatives of even $\sigma$-orbits we  take
\[
 \gf{1}{0}\,0 \cdots 0\,\gf{1}{0} \qquad \quad
 \gf{0}{1}\,0 \cdots 0\,\gf{0}{1} \qquad \quad
 \gf{2}{0}\,0 \cdots 0\, \gf{0}{0} \qquad \quad
 \gf{0}{2}\,0 \cdots 0\,\gf{0}{0}
\]
and for each integer $j$ with $0<2j\le k$, the 2-labeling  with $1$ at $2j$.
Thus
\[h^\even(\DD_{2k})=\lfloor k/2\rfloor+4.\]

We compute $\Orb^\odd(\DD_\ell)$.
For representatives of odd $\sigma$-orbits we  take
\[
\gf{1}{1}\,0 \cdots 0\,\gf{0}{0} \qquad\quad
\gf{1}{0}\,0 \cdots 0\, \gf{0}{1}
\]
and for each integer $j$ with $1<2j+1\le k$, the 2-labeling  with $1$ at $2j+1$.
Thus
\[ h^\odd(\DD_{2k})=\lceil k/2\rceil+1.\]

As an example, we give a list of representatives of even and odd orbits for $\DD_6$:
\begin{align*}
\Orb^\even(\DD_6):\qquad &\gf{1}{0}\, 0\,0\,0\, \gf{1}{0}\qquad \gf{0}{1}\, 0\, 0\,0\, \gf{0}{1}\qquad \gf{2}{0}\, 0\,0\,0\, \gf{0}{0}
\qquad \gf{0}{2}\, 0\,0\,0\, \gf{0}{0}\qquad \gf{0}{0}\, 1\,0\,0\, \gf{0}{0}\ ,\\
&\\
\Orb^\odd(\DD_6):\qquad &\gf{1}{1}\, 0\,0\,0\, \gf{0}{0}\qquad \gf{1}{0}\, 0\,0\,0\, \gf{0}{1}\qquad \gf{0}{0}\, 0\,1\,0\, \gf{0}{0}\ .
\end{align*}

Note that if $\ell>4$, our compact half-spin group $\GG$ has no outer automorphisms, hence all its real forms are {\em inner}  forms,
and we have computed the Galois cohomology for all the forms of $\GG$.

Note also that for the compact half-spin group $\GG$ we have
\[ \# H^1(\R,\GG)=h^\even(\DD_{2k})=\lfloor k/2\rfloor+4=\lfloor \ell/4\rfloor+4. \]
For comparison, $\#H^1(\R,\SO_{2\ell})=\ell+1$.
We have $\lfloor \ell/4\rfloor+4=\ell+1$ for an even natural number $\ell$ if and only if $\ell=4$.
(In this case, because of triality, our half-spin group $\GG$ is isomorphic to $\SO_8$.)


\end{document}